\documentclass[a4paper,11pt]{article}
\usepackage{amssymb,amsmath,amsthm,amscd}
\usepackage{graphics,graphicx,color}
\usepackage{latexsym}
\usepackage{srcltx, bm}
\DeclareGraphicsExtensions{.eps} \textwidth=169mm
\definecolor{darkred}{rgb}{0.9,0.1,0.1}

\newtheorem{theorem}{Theorem}[section]
\newtheorem{lemma}{Lemma}[section]

\newtheorem{remark}{Remark}[section]
{\rm}
\definecolor{darkred}{rgb}{0.9,0.1,0.1}

\definecolor{darkred}{rgb}{0.9,0.1,0.1}

\begin{document}

\title{H\"older estimates for solutions of parabolic SPDEs.}

\setcounter{footnote}{2}
\author{S.B. Kuksin\thanks{
Institut de Math\'emathiques de Jussieu--Paris Rive Gauche, CNRS, Universit\'e Paris Diderot, UMR 7586, Sorbonne Paris Cit\'e, F-75013, Paris, France \& School of Mathematics, Shandong University, Jinan, Shandong,  China,
{\tt Sergei.Kuksin@imj-prg.fr}}, \ N.S. Nadirashvili\thanks{
Aix-Marseille Universit\'e, 39, rue F. Joliot-Curie, 13453 Marseille, France,
{\tt  nikolay.nadirashvili@univ-amu.fr}
}
\  and A.L. Piatnitski\thanks{{The Arctic University of Norway, campus Narvik, P.O.Box 385, 8505 Narvik, Norway
\&
Institute for Information Transmission Problems of RAS, Bolshoi Karetny per. 19, 127051, Moscow, Russia},
{\tt apiatnitski@gmail.com}
} 
}
%
%

\date{}

\maketitle

\begin{abstract}
This paper considers second-order stochastic partial differential equations with additive
noise given in a bounded domain of $\mathbb R^n$. We suppose that the coefficients of the noise are
$L^p$-functions with sufficiently large $p$. We prove that the solutions are H\"older-continuous functions
almost surely (a.s.) and that the respective H\"older norms have finite momenta of any order.
\end{abstract}



\section{Introduction}
This paper is devoted to $L^\infty$-
 and H\"older estimates for solutions
of second-order parabolic stochastic partial differential equations (SPDEs) in bounded domains,
under the minimal regularity assumptions on the coefficients. In recent years several
works focused on the $L^p$- and $W^{l,p}$-theory for parabolic SPDEs; however, the $W^{l,p}$-estimates
obtained in these works rely on rather strong regularity assumptions on the coefficients. On
the other hand, in various qualitative and applied problems the estimates in terms of $L^\infty$-
or even $L^p$-norms of the coefficients are very important. Major progress in this area was
achieved by Krylov in \cite{Kry}, where the regularity assumptions were reduced essentially for
equations posed in the whole $R^d$.

In our work we propose another technique which allows us to further reduce the regularity
assumptions and to treat initial-boundary problems for solutions of SPDEs with additive
noise. Although this class of equations is smaller than that in \cite{Kry}, still it is sufficiently large
and contains many important operators. Our approach differs essentially from that in \cite{Kry}
and gives rise to a shorter proof of the main results.

In this paper we study SPDEs of the form
\begin{equation}\label{ori_spde}
\begin{array}{c}
\displaystyle
  du = Au dt +\sum\limits_{j=1}^N f^j(x,t) dw_t^j,\qquad (x,t)\in Q\times(0,+\infty),\\
  \displaystyle
  u|_{t=0}=0,\qquad u|_{\partial Q}=0,
  \end{array}
\end{equation}
where $Q$ is a smooth bounded domain in $\mathbb R^n$, $w^j_t$
 are standard independent one-dimensional
Brownian motions, and $f^j(x, t)$ are progressively measurable $L^\infty$- 
or $L^p$-functions; in the
latter case $p$ is sufficiently large. Regarding the operator $A$ we assume that $A$ is either a
generic second-order uniformly elliptic operator with sufficiently smooth coefficients or is a
divergence form uniformly elliptic operator with measurable bounded coefficients.
In this paper we prove that a solution of \eqref{ori_spde} is a H\"older-continuous function a.s. and
that the corresponding $L^\infty$- 
and H\"older norms have finite momenta of any order.

For a stochastically perturbed heat equation (i.e., for \eqref{ori_spde} with $A = \Delta$) with periodic
boundary conditions these estimates were proved by one of the authors in the appendix of \cite{Kuks}.
It should be noted that the results of this paper can be generalized to the equations
involving the infinite-dimensional stochastic term
$\sum_{j=1}^\infty f^j(x,t)dw_t^j$
instead of the finite-dimensional
one; $L^\infty$-convergence (or $L^p$-convergence with sufficiently large $p$) of the sum 
$\sum_{j=1}^\infty |f^j(x,t)|$ is then required.

On the basis of the $L^\infty$-estimates obtained in this work, one can examine the regularity
of solutions of some nonlinear SPDEs and prove the solvability of these equations in certain
functional spaces (see \cite{Kuks}). When analyzing SPDEs, the $L^\infty$-estimates of this type can
be used in the same way as the classical maximum principle in the case of deterministic
equations.

In the first part of the paper we obtain several estimates for H\"older and $L^\infty$-norms
of solutions of deterministic parabolic equations. In what follows we make use of these
estimates to derive point-wise bounds for increments of the solution of \eqref{ori_spde} and to estimate
the corresponding function norms.

\section{ Setting of the problem. Technical estimates.}
\label{sec_setting}

Let $w^j_t$  be standard independent
Brownian motions defined on a probability space $(\Omega,\mathcal{F},\mathbf{P})$,  and let $(\mathcal{F}_t)$ be the
corresponding filtration of $\sigma$-algebras. We assume that $f^j(x, t)$ are $\mathcal{F}\times \mathcal{B}$-measurable,
 $(\mathcal{F}_t)$-adapted functions, where $\mathcal{B}$ stands for the Borel $\sigma$-algebra, and that at least one of the
following two assumptions holds:
$$
\begin{array}{lrrlll}
\mathbf{B}^\infty&\!\!\!:&\ \ |f(x, t)|&\!\!\! \leqslant 1 &\hbox{ almost surely (a.s.) for all $x$ and $t$},\\[2mm]
\mathbf{B}^p&\!\!\!:&\ \  \|f(·, t)\|_{L^p(Q)}&\!\!\! \leqslant 1 &\hbox{ a.s. for every } t.
\end{array}
$$
In the following two statements, $A$ is a uniformly elliptic second-order differential operator
of the form
$$
A = a_{ij}(x)\frac\partial{\partial x_i}\frac\partial{\partial x_j}
+ b_i(x))\frac\partial{\partial x_i} + c(x)
$$
with smooth coefficients.

\begin{theorem}\label{thm1}
Let $\mathbf{B}^\infty$ hold. Then under the above conditions\\
(a) for any $k\geqslant 0$ and any $T\geqslant 0$, a solution of problem \eqref{ori_spde} satisfies the estimate
\begin{equation}\label{est1}
  \mathbf{E}\|u\|_{L^\infty(Q_T)}^k\leqslant C_k\exp(\overline{c} T),
\end{equation}
where $Q_T = Q \times [T,T + 1]$ and $\overline{c} = \max(0, \max_Q c(x))$;\\
(b) the solution $u(x, t)$ is H\"older-continuous a.s., and for any $T\geqslant 0$, $k \geqslant 0$, and $\theta < \frac12$
the upper bound
\begin{equation}\label{est2}
\mathbf{E}
\|u\|_{C^\theta(Q_T)}^k
\leqslant C(k,\theta) \exp(\overline{c}T)
\end{equation}
holds.
\end{theorem}

In estimates \eqref{est1}, \eqref{est2}, and everywhere below, for brevity we write $|u|_{L^\infty(Q_T )}$ instead of
$|(u |_{Q_T} )|_{L^\infty(Q_T )}$, etc.
\begin{theorem}\label{thm2}
 If $\mathbf{B}^p$ holds for some $p > d$, then\\
(a) for any $k \geqslant 0$ and any $T \geqslant 0$, a solution of problem \eqref{ori_spde} satisfies the estimate \eqref{est1};\\
(b) the solution $u(x, t)$ is H\"older-continuous a.s., and for any $T \geqslant 0$, $k \geqslant 0$, and
$\theta < 1/2-d/(2p)$ the upper bound \eqref{est2} holds.
\end{theorem}
\begin{remark}{\rm
If the zero-order term $c(x)$ of the operator $A$ is non-positive, then the exponential
factor disappears in all the estimates \eqref{est1}--\eqref{est2} above.
The case of a generic operator that satisfies our assumptions can be reduced to the case
from Remark 1 by means of the following change of the unknown function $U = \exp(-\alpha t)u$,
with a proper choice of $\alpha$.}
\end{remark}

From now on we assume without loss of generality that $c(x)\leqslant 0$.

For a uniformly elliptic divergence form operator $A$ with bounded measurable coefficients
one has the following result.
\begin{theorem}\label{thm3}
  Let $A$ be a uniformly elliptic divergence form operator with bounded coefficients.
Then under the assumption $\mathbf{B}^\infty$,\\
(a) for any $k\geqslant 0$ and any $T\geqslant 0$, the solution of problem \eqref{ori_spde} satisfies the estimate
$$
\mathbf{E}\,\|u\|^k_{L^\infty(Q_T)}\leqslant C_k;
$$
(b) the solution $u(x, t)$ is a.s. H\"older-continuous, and there is $\alpha > 0$ depending only on
the lower and upper ellipticity constants of $A$ and on the geometry of the domain $Q$ such that
for any $T \geqslant 0$, $k \geqslant 0$, and $\theta < \alpha$ the upper bound
\begin{equation}\label{est3}
\mathbf{E}
\|u\|^k_{C^\theta(Q_T)}\leqslant C(k,\theta)
\end{equation}
holds.
Under the assumption $\mathbf{B}^p$ with $p > d$,\\
(c) there exists $\alpha_0 > 0$ such that estimate \eqref{est3} holds true for every $T \geqslant 0$, $k \geqslant 0$,
and $\theta \leqslant \alpha_0$.
\end{theorem}

The proof of these theorems relies on several auxiliary statements. Let us consider the
following initial-boundary problem:
\begin{equation}\label{auxpbm1}
  \frac{\partial v}{\partial t}-Av=0,\qquad v\big|_{t=0}=F(x),\qquad v\big|_{\partial Q}=0,
\end{equation}
and denote by $S_t$ the corresponding semigroup:
$$
v(\cdot, t) = (S_tF)(\cdot), \qquad F= F(\cdot).
$$
The following statement deals with a “smooth” operator $A$.
\begin{lemma}\label{lem1}
 Let $|F(x)| < M$. Then, for any $\theta < 1$, the following estimates hold with
$c > 0$:
\begin{eqnarray}
\label{eqa1}
  \|v(\cdot,t)\|_{C^\theta(Q)} &\leqslant & c(\theta) M t^{-\frac\theta2}\exp(-ct), \\
  \label{eqa2}
  |v(x,t+\delta)-v(x,t)| &\leqslant & c(\theta) M t^{-\theta}\delta^\theta \exp(-ct).
\end{eqnarray}
Moreover, if  $\|F\|_{L^q(Q)} \leqslant M$ and $q > 1$, then
\begin{eqnarray}
\label{eqa3}
  \|v(\cdot,t)\|_{C^\theta(Q)} &\leqslant & c(\theta) M t^{-\frac\theta2-\frac d{2q}}\exp(-ct), \\
  \label{eqa4}
  |v(x,t+\delta)-v(x,t)| &\leqslant & c(\theta) M t^{-\theta-\frac d{2q}}\delta^\theta \exp(-ct).
\end{eqnarray}
\end{lemma}
\begin{proof}
Denoting by $G(x, y, t)$ the Green function associated with problem \eqref{auxpbm1}, we have
\begin{equation}\label{est_fund1}
  \|G(x, \cdot, t)\|_{L^p(Q)}\leqslant c(p)t^{-d/(2q)},\qquad \frac1p+\frac1q=1
\end{equation}
(see, for example, \cite{Fried}). This implies the inequality
\begin{equation}\label{est_fund2}
\|v(\cdot, t)\|_{L^\infty(\mathbb R^d)} \leqslant c(p) t^{-d/(2q)}
\|F\|_{L^q(\mathbb R^d)}
\end{equation}
for any $F \in L^q(\mathbb R^d)$.

In addition to problem \eqref{auxpbm1}, we consider in the domains $Q_\gamma = \gamma Q = \{\gamma x: x \in Q\}$, 
$\gamma \geqslant 1$, the initial-boundary problems
\begin{equation}\label{auxpbm2}
\begin{array}{c}
\displaystyle
\frac{\partial V}{\partial t}-  A_\gamma V = 0, \qquad x\in Q_\gamma,\\[2mm]
V\big|_{t=0}= F(\gamma x),\qquad  V\big|_{\partial Q_\gamma}= 0,
\end{array}
\end{equation}
where
$$
A_\gamma = a_{ij}(\gamma x)
\frac\partial{\partial x_i}\frac\partial{\partial x_j}
+ \gamma^{-1} b_i(\gamma x)
\frac\partial{\partial x_i}+ \gamma^{-2}
c(\gamma x).
$$
Due to the smoothness of the coefficients of $A$ and of the domain $Q$, by the standard elliptic
estimates one has
\begin{equation}\label{est6}
  \left|
\frac\partial{\partial t} V (x, 1)
\right|+
\left| \nabla_xV (x,1)\right|\leqslant
c \left\|V(\cdot,\frac12)\right\|_{L^\infty(Q_\gamma)}
\end{equation}
uniformly in $x \in Q_\gamma$, with the constant $c$ that does not depend on $\gamma$ (see \cite{LSU}, \cite{GT}).
After an appropriate rescaling, the last estimate reads
\begin{equation}\label{est7}
t\left| \frac\partial{\partial t}v(x,t)\right| + t^\frac12\left|\nabla_xv(x,t)\right| \leqslant
c \Big\| v\big(\cdot,\frac t2\big)\Big\|_{L^\infty(Q)}
\end{equation}
for all $t$, $0 < t \leqslant 1$. From \eqref{est_fund2} and \eqref{est7} it follows that
\begin{equation}\label{est8}
\left| \frac\partial{\partial t}v(x, t)\right|
\leqslant ct^{-1-\frac d{2q}}
 \|F\|_{L^q(Q)},
\end{equation}
\begin{equation}\label{est9}
\left| \frac\partial{\partial x}v(x, t)\right|
\leqslant ct^{-\frac12-\frac d{2q}}
 \|F\|_{L^q(Q)},
\end{equation}
for all $t \leqslant 1$. Finally, we use the interpolation inequality\footnote{Denoting $\|v\|_{L^\infty}$ by $\gamma_0$ and $\|v\|_{C^1}$ by
$\gamma_1$, we have $|v(x+y)-v(x)|\leqslant 2\gamma_0$  and     $|v(x+y)-v(x)|\leqslant 2\gamma_1|y|$. Raising these inequalities to the degrees $(1 - \theta)$ and $\theta$, respectively, and multiplying the new
ones, we get that $|v(x + y) - v(x)| \leqslant 2^{1-\theta} \gamma_0^{1-\theta}\gamma_1^\theta |y|^\theta$.
This relation implies the required inequality.}
$$
\|v\|_{C^\theta}\leqslant 2\|v\|^{1-\theta}_{L^\infty}\|v\|^\theta_{C^1}
$$
to derive \eqref{eqa1}, \eqref{eqa2}, and \eqref{eqa3}, \eqref{eqa4} from \eqref{est_fund2},  \eqref{est8}, and \eqref{est9}. For large $t$ the required estimates follow from the exponential decay of the Green function. This completes the proof of the lemma.
\end{proof}

For operators $A$ with measurable coefficients we have the following result.
\begin{lemma}\label{lem2}
 Let $A$ be a divergence form uniformly elliptic operator with bounded coefficients.
Then there exists $\alpha > 0$ depending only on the ellipticity constants of $A$ and on the
domain $Q$, such that for any $\theta < \alpha$ a solution $v$ of problem \eqref{auxpbm1}  admits all the estimates of
Lemma \ref{lem1}.
\end{lemma}

\begin{proof}
The proof is rather similar to that of the preceding lemma and we outline it
briefly. According to \cite{Aro} the inequality \eqref{est_fund1} (and thus  \eqref{est_fund2}) is still valid. The bound  \eqref{est6} is
not valid any more but instead we can use the Nash-type estimates for the H\"older norms of
the solution $V (x, t)$ (see \cite{GT}, \cite{LSU}) to get that
$$
\|V\|_{C^\alpha(Q_\gamma\times[1,2])}  \leqslant c
\left\|V\big(\cdot,\frac12\big)\right\|_{L^\infty(Q)}
$$
for some $\alpha > 0$ and $c$ independent of $\gamma$. The rest of the proof is the same as in the previous
lemma.
\end{proof}

Now we proceed with problem \eqref{ori_spde} in a “smooth” case. Its solution does exist, is unique,
and can be written in the form (see \cite[Chapter 5]{DaPraZa} and \cite{Kry})
$$
u(t) =\int_0^t \sum\limits_{j=1}^N  S_{t-\tau} f^j(x,\tau) dw_\tau.
$$
The last formula is, in fact, the definition of a {\sl mild solution}. Clearly it is sufficient to
consider one element of the sum above. In what follows the summation sign and the index $j$
will be omitted.
We begin by obtaining point-wise estimates of increments of the solution
in $x$ variables. To this end we denote $[S_{t-\tau} f(x_1, \tau) - S_{t-\tau} f(x_2, \tau)]$ by $g(t, \tau )$ and introduce
a random variable
$$
U = u(t, x_1) - u(t, x_2) =\int_0^t g(t,\tau)dw_\tau.
$$
The quadratic characteristics of the stochastic integral is given by $X(t) =\int_0^t
g^2(t, \tau ) d\tau$.
From Lemma \ref{lem1} we derive the estimate
$$
X(t) \leqslant M^2c(\theta) |x1 - x2|^{2\theta}\int_0^t t^{-\theta}\exp(-2ct)\,dt
\leqslant  M^2c'(\theta) |x_1 - x_2|^{2\theta},
$$
valid for any $\theta < 1$. Thus, by the Burkholder–Davis–Gundy inequality we have
\begin{equation}\label{est10}
|U|^p \leqslant c_1(p)\mathbf{E}(X^{p/2}) \leqslant c(p,\theta)M^p|x_1 - x_2|^{p\theta},\qquad \theta < 1.
\end{equation}
Similarly, for a time increment we define
$$
\begin{array}{rl}
 \displaystyle
R =\!\!&\!\! u(x, t + \delta) - u(x, t) =
 \displaystyle
\int_t^{t+\delta} S_{t+\delta-\tau} f(x, \tau ) dw_\tau\\[2mm]
&\displaystyle
+\int_0^t \big(S_{t+\delta-\tau} f(x, \tau ) - S_{t-\tau} f(x, \tau )\big)dw_\tau = I_1 + I_2.
\end{array}
$$
The first integral $I_1$ can be estimated as follows:
\begin{equation}\label{est11}
\mathbf{E}|I_1|^p \leqslant c\delta^{p/2}M^p.
\end{equation}
Indeed, if we denote $S_{t+\delta-\tau} f(x, \tau)$ by $h_1(t, \tau )$, then
$$
\int_t^{t+\delta}
h^2_1(t, \tau ) \tau \leqslant M^2\delta,
$$
and required estimate \eqref{est11} is a direct consequence of the Burkholder–Davis–Gundy inequality.

In order to estimate $I_2$, we define
$$
h_2(t, \tau) = S_{t+\delta-\tau} f(x, \tau ) - S_{t-\tau} f(x, \tau ).
$$
By Lemma \ref{lem1} we have
$$
|h_2(t, \tau )| \leqslant  c(\theta)M(t - \tau )^{-\theta}\delta^\theta \exp\big(-c(t - \tau )\big)
$$
for any $\theta < \frac12$. Thus,
$$
\int_0^th_2^2(t, \tau) d\tau \leqslant M^2\delta^{2\theta}\int_0^t
(t - \tau)^{-2\theta} \exp\big(-2c(t - \tau )\big) d\tau \leqslant c(\theta)M^2 \delta^{2\theta}
$$
and by the Burkholder–Davis–Gundy inequality we have
\begin{equation}\label{est14}
\mathbf{E}|I_2|^p \leqslant c(\theta) \delta^{\theta p}M^p, \qquad \theta< \frac12.
\end{equation}
Combining  \eqref{est10},  \eqref{est11}, and \eqref{est14}, we obtain
\begin{equation}\label{est15}
  \mathbf{E}
\left|
u(x_1, t_1) - u(x_2, t_2)\right|^p
 \leqslant c(\theta)
	\big(|t_1 - t_2| + |x_1 - x_2|\big)^{\theta p}
\end{equation}
for any $p > 1$ and any $\theta < \frac 12$.

\section{ Boundedness and H\"older-continuity of the solution; higher momenta.}
\label{sec3}

This section deals with the estimation of $L^\infty$
and $C^\theta$  norms of the solution of problem \eqref{ori_spde}.
The first statement here is given by the following lemma.
\begin{lemma}\label{lem3}
  Let the functions $f^j(x, t)$ satisfy condition $\mathbf{B}^\infty$
of the preceding section. Then
\begin{itemize}
  \item [{\rm (a)}] there is a.s. a continuous version of the solution $u(x, t)$ of problem \eqref{ori_spde};
  \item [{\rm (b)}] under this choice of $u(x, t)$, for any $T \geqslant 0$ and any $p \geqslant 1$, we have
$$
\mathbf{E}\|u\|^p_{L^\infty(Q\times[T,T+1])} \leqslant c(p) < \infty.
$$
\end{itemize}
\end{lemma}

\begin{proof}
 The first statement follows immediately from \eqref{est15} by the Kolmogorov theorem;
it just suffices to choose a sufficiently large $p$.

To prove the second one we consider a cylinder $Q^T = Q\times[T,T +1]$ and define a sequence
of sets
$$
\mathcal{K}_n =\{
k \in \mathbb Z^{d+1}
\big|
k2^{-n} \in Q^T\},\qquad
n= 1, 2, \ldots.
$$
For an arbitrary vector $e \in \mathbb Z^{d+1}$ such that
$$
|e|_1 \sim \max\limits_
{1\leqslant j\leqslant d+1}|e_j | = 1,
$$
we then set $\xi_k^{n,e} = |u((k + e) 2^{-n}) - u(k2^{-n})|$. By Lemma \ref{lem1} we have
\begin{equation}\label{est18}
\mathbf{E}|\xi_k^{n,e}|^p \leqslant c(\theta) 2^{-np\theta},\qquad 0 < \theta <\frac12,
\end{equation}
for any $p \geqslant 1$. Let us introduce the events
$$
\mathcal{A}^{n,e}_{k,q} =
\left\{\omega\in \Omega \,|\, \xi_k^{n,e}\geqslant Kq^n\right\}.
$$
From \eqref{est18} by the Chebyshev inequality we get
$$
\mathbf{P}(\mathcal{A}^{n,e}_{k,q})\leqslant \frac{\mathbf{E}|\xi_k^{n,e}|^p}{K^p q^{np}}
\leqslant  c(\theta)\frac{2^{-np\theta}}{K^p q^{np}}.
$$
For each $n$ the total number of the events $\mathcal{A}^{n,e}_{k,q}$
 is not greater than $2^{\bar c(d,Q)n}$. Thus the
probability of the union $\mathcal{A}^n_q=\bigcup_{k\in\mathcal{K}}\big(\bigcup_{|e|_1=1}\mathcal{A}^{n,e}_{k,q}\big)$
satisfies the estimate
$$
\mathbf{P}(\mathcal{A}^n_q) \leqslant c(\theta)K^{-p}q^{-pn}
2^{\bar c(d,Q) n-\theta pn} = c(\theta)K^{-p}
\alpha^n,
$$
where $\alpha = 2^{\bar c(d,Q)}/(2^\theta q)^p$. If we set $\theta = \frac13$,
 $ q = 2^{-1/6}$, then $\alpha = 2\bar c(d,Q)/2^{p/6}$. Finally, taking
$p > 6\bar c(d,Q) + 1$, we have $\alpha < \frac12$.
Hence, for the probability of the event $\mathcal{A} =\bigcup_{n>1}\mathcal{A}^n_q$
the following estimate holds true:
\begin{equation}\label{est30}
\mathbf{P}(\mathcal{A}) \leqslant c(\theta)K^{-p}.
\end{equation}

From now on we assume without loss of generality that $Q$ is a $d$-dimensional cube $(0, 1)^d$.
Indeed, to reduce the case of a general domain to that of a cube it suffices to extend the
function $u(x, t)$ in the exterior of $Q$ as zero and to rescale the arguments if necessary.

Next, any point of $Q$ can be represented in the form
$$
x =\sum\limits_{j=1}^\infty
e(j) 2^{-j},\qquad |e(j)|_1 \leqslant 1.
$$
Let us denote $x(m) =
\sum_{j=1}^m e(j) 2^{-j}$
 (as usual, $x(0) = 0$). Clearly, $u(x(0), t) = 0$. Then, by the definition of the set $\mathcal{A}^{n,e}_{k,q}$,
  for any $\omega \not\in \mathcal{A}$ we have
$$
\big|
u(x(m), t) - u(x(m + 1), t)\big|
\leqslant Kq^m = K2^{-m/6}.
$$
Therefore,
\begin{equation}\label{est31}
|u(x, t)| \leqslant K
\sum\limits_{m=1}^\infty 
2^{-m/6} = K2^{1/6}(2^{1/6} - 1).
\end{equation}
It now follows from \eqref{est30} and \eqref{est31} that
$$
\mathbf{P}
\big\{
\|u\|_{L^\infty(Q^T)}
 \geqslant K\big\}
\leqslant c_p(K + 1)^{-p}
$$
for any $p > \bar c(d,Q)$. Finally, denoting $U_T = \|u\|_{L^\infty(Q^T)}$, we find that
$$
\mathbf{E}U^s_T =
\int_0^\infty x^s d\mathbf{P}
\{U_t \leqslant x\} = s
\int_0^\infty  x^{s-1}\mathbf{P}\{U_t \leqslant x\}dx
$$
$$
\leqslant 
 sc_p
\int_0^\infty  x^{s-1}(x + 1)^{-p} dx < \infty
$$
if $p > s + 1$. The lemma is proved.
\end{proof}

In order to estimate the H\"older norms of $u(x, t)$, let us first formulate the following
simple assertion.
\begin{lemma}\label{lem4}
Let a function $g(y)$ satisfy the estimate
\begin{equation}\label{est_cond}
  \mathop{\mathrm{osc}}\limits_J g(\cdot)\leqslant  \gamma_n
\end{equation}
in any small cube $J$ which is a mesh of the grid $2^{-n} \mathbb Z^{d+1}$, i.e., in any 
$J = 2^{-n}j +[0, 2^{-n}]^{d+1}$,
where $j \in \mathbb Z^{d+1}$. Then, for any ${\scriptstyle\Delta} \in \mathbb R^{d+1}$, one has
$$
\big| g(y +{\scriptstyle\Delta}) - g(y)\big|\leqslant 2\gamma\big._{[\log_2(1/|{\scriptstyle\Delta}|)]},
$$
where $[\cdot]$ stands for the integer part.
\end{lemma}
\begin{proof}
For $n = [\log_2(1/|{\scriptstyle\Delta}|)]$ let us consider the family of closed cubes defined above.
Clearly, both $y$ and $y +{\scriptstyle\Delta}$  either belong to the same cube or they are situated in adjacent
cubes having at least one point in common. So the statement of the lemma follows.
\end{proof}

Next, for a solution of \eqref{ori_spde} and for any $\omega \not\in \mathcal{A}$, inequality \eqref{est_cond} holds with $\gamma_n = Kq^n$.
Thus,
\begin{eqnarray*}
\big|u((x, t) + {\scriptstyle\Delta}) - u(x, t)\big|
&\leqslant& 2Kq^{[\log_2(1/|\Delta|)]} \leqslant 2Kq^{\log_ 2(1/|\Delta|)-1}\\
&=& 2Kq^{-1} 2^{\log_2(q) \log_2(1/|\Delta|)} = 2Kq^{-1}|{\scriptstyle\Delta}|^{\log_2(1/q)}.
\end{eqnarray*}
If we put $q = 2^{-\theta_1}$  with $\theta_1 < \theta < \frac12$, then $\log_2(1/q) = \theta_1$, and the above inequality implies
$$
\big|u((x, t) + {\scriptstyle\Delta}) - u(x, t)\big|
\leqslant 2K2^{\theta_1}|{\scriptstyle\Delta}|^{\theta_1}  \leqslant 4K|{\scriptstyle\Delta}|^{\theta_1}.
$$
Therefore, for all $\omega \not\in \mathcal{A}$,
$$
\|u\|_{C^{\theta_1}(Q^T)} \leqslant 4K.
$$
Taking into account \eqref{est30} we derive from the last estimate that
$$
\mathbf{P}\left\{\|u\|_{C^{\theta_1}(Q^T)} > K\right\}\leqslant c(p, \theta_1)K^{-p}
$$
for all $p > 1$. Finally, for any $\theta_1 < \frac12$ and any $m \geqslant 1$ the bound
$$
\mathbf{E}\|u\|^m_{C^{\theta_1}(Q^T)} \leqslant c(m, \theta_1)
$$
can be proved in the usual way, and we arrive at the following assertion.
\begin{lemma}\label{lem5}
For any $\theta < \frac12$ a solution $u(x, t)$ of problem \eqref{ori_spde} is a.s. a $C^\theta$-function.
Moreover, for any cylinder $Q^T$ and any $m \geqslant 1$, the inequality $\mathbf{E}\|u\|^m_{C^\theta(Q^T)} \leqslant c(m, \theta)$ holds.
\end{lemma}
Now, to complete the proof of Theorem \ref{thm1} it is merely sufficient to refer to Lemmas \ref{lem3}
and \ref{lem5}.

Other statements of this work can be justified in the same way.

\end{document}